\documentclass[12pt]{article}
\usepackage{mathrsfs}
\usepackage{amsmath}
\usepackage{amssymb}
\usepackage{amsfonts}
\usepackage{wasysym}
\usepackage{verbatim}
\usepackage{graphicx}
\usepackage[amsmath,thmmarks]{ntheorem}

\def\F{\mathscr{F} }
\def\S{\mathcal{S} }
\def\N{\mathcal{N} }

\def\R{\mathbb{R}}

\def\C{\mathbb{C}}

\def\ba{\begin{array}}
\def\ea{\end{array}}
\def\be{\begin{enumerate}}
\def\ee{\end{enumerate}}
\def\bi{\begin{itemize}}
\def\ei{\end{itemize}}
\def\bc{\begin{cases}}
\def\ec{\end{cases}}
\def\beq{\begin{equation}}
\def\eeq{\end{equation}}
\def\beqs{\begin{equation*}}
\def\eeqs{\end{equation*}}
\def\beqa{\begin{eqnarray}}
\def\eeqa{\end{eqnarray}}
\def\beqas{\begin{eqnarray*}}
\def\eeqas{\end{eqnarray*}}
\def\bmul{\begin{multline}}
\def\emul{\end{multline}}
\def\bmuls{\begin{multline*}}
\def\emuls{\end{multline*}}
\def\bg{\begin{gather}}
\def\eg{\end{gather}}
\def\bgs{\begin{gather*}}
\def\egs{\end{gather*}}

\theorembodyfont{\normalfont\sl}
\newtheorem{thm}{Theorem}[section]
\newtheorem{cor}[thm]{Corollary}
\newtheorem{lem}[thm]{Lemma}
\newtheorem{prop}[thm]{Proposition}

\numberwithin{equation}{section}

\begin{document}

\baselineskip=24pt \thispagestyle{empty}

\newpage

\setcounter{page}{1}

\title{
\baselineskip=24pt
 \bf On a quadratic nonlinear Schr\"{o}dinger \\[.5ex]
 \bf equation:  sharp well-posedness\\[.5ex]
 \bf and ill-posedness\footnote{This work is
 supported by National Natural Science Foundation of China
 under grant numbers 10471047 and 10771074.}}

\bigskip

\author{
\baselineskip=24pt
 \normalsize Yongsheng Li and Yifei Wu\footnote{
             Email: yshli@scut.edu.cn (Y. S. Li) and
             yerfmath@yahoo.cn (Y. F. Wu)}\\[1ex]
 \normalsize Department of Mathematics,
             South China University of Technology, \\[1ex]
 \normalsize Guangzhou, Guangdong 510640, P. R. China
 }

\bigskip

\date{}

\bigskip

\bigskip

\leftskip 1cm \rightskip 1cm

\maketitle

\noindent
 {\small{\bf Abstract}\quad
 \baselineskip=24pt
We study the initial value problem of  the quadratic nonlinear
Schr\"{o}dinger equation
$$
iu_t+u_{xx}=u\bar{u},
$$
where $u:\R\times \R\rightarrow \C$. We prove that it's locally
well-posed in $H^s(\R)$ when $s\geq -\dfrac{1}{4}$ and ill-posed
when $s< -\dfrac{1}{4}$, which improve the previous work in
\cite{KPV}. Moreover, we consider the problem in the following
space,
$$
H^{s,a}(\R)=\left\{u:\|u\|_{H^{s,a}}\triangleq
\left(\displaystyle\int
\left(|\xi|^s\chi_{\{|\xi|>1\}}+|\xi|^a\chi_{\{|\xi|\leq
1\}}\right)^2|\hat{u}(\xi)|^2\,d\xi\right)^{\frac{1}{2}}<\infty\right\}
$$
for $s\leq 0, a\geq 0$. We establish the  local well-posedness in
$H^{s,a}(\R)$ when $s\geq -\dfrac{1}{4}-\dfrac{1}{2}a$ and
$a<\dfrac{1}{2}$. Also we prove that it's ill-posed in $H^{s,a}(\R)$
when $s<-\dfrac{1}{4}-\dfrac{1}{2}a$ or $a>\dfrac{1}{2}$. It remains
the cases on the line segment: $a=\dfrac{1}{2}$, $-\dfrac{1}{2}\leq
s\leq 0$ open in this paper.}
\bigskip

\bigskip

\noindent
 {\small{\bf  MR(2000) Subject Classification:}\quad 35Q55}

\bigskip

\bigskip

\leftskip 0cm \rightskip 0cm

\normalsize

\baselineskip=24pt

\section{Introduction}
This paper is concerned with the low regularity behavior of the
initial value problem (IVP) for 1-D quadratic nonlinear
Schr\"{o}dinger equations
\renewcommand{\arraystretch}{2}
\begin{eqnarray}\label{QNLS}
  && iu_t+u_{xx}=Q(u,\bar{u}),
     \qquad x,\,t\in \R,\\
  && u(x,0)=u_0(x),\label{1.2}
\end{eqnarray}
where $Q: \C^2\rightarrow \C$ is a quadratic polynomial. This
particular problem as well as its higher dimensional version, has
been extensively studied. Here, we refer some of them, which are
closely related to our topic. As it's well-known, the IVP of
(\ref{QNLS}) is locally well-posed in $H^s(\R)$ when $s\geq 0$ for
any type quadratic nonlinearity, see \cite{CW} and \cite{T}. The
results were proved by the Strichartz estimates. It's sharp in some
sence, because the IVP (\ref{QNLS}) is ill-posed when $s<0$ if the
nonlinearity is $|u|u$ (power type) (see \cite{KPV2} for focusing
case, and \cite{CCT} for defocusing case), by Gallilean invariance.
However, it's shown by Kenig, Ponce and Vega in \cite{KPV} that, one
can lower the regularity below $s=0$ if the nonlinearity is not
Gallilean invariance. Three typical nonlinearities  of this type are
\begin{equation}
Q(u,\bar{u})= u^2, u\bar{u}, \bar{u}^2. \label{Q}
\end{equation}
In \cite{KPV}, the authors established the local well-posedness for
$s>-\dfrac{3}{4}$ if the nonlinearity is of $u^2$ or $\bar{u}^2$
type, and for $s>-\dfrac{1}{4}$ if it is $u\bar{u}$. The results
were proved by the Bourgain argument (see \cite{Bourgain} and
\cite{KPV1}), which were mainly based on a bilinear estimate in
Bourgain space $X_{s,b}$. On the other hand, there are
counterexamples shown in \cite{KPV} and \cite{NTT} that the key
bilinear estimates in \cite{KPV} fail to hold in $X_{s,b}$, when
$s\leq -\dfrac{3}{4}$ for $u^2$, $\bar{u}^2$, and $s\leq
-\dfrac{1}{4}$ for $u\bar{u}$. It suggests that the common Bourgain
space is not sufficient to study (\ref{QNLS})(\ref{Q}) in a lower
regular space. However, it doesn't mean that it's not well-posed in
$H^s(\R)$ of some lower indices. Indeed, in \cite{BT}, Bejenaru and
Tao pushed the threshold to $s\geq -1$ when the nonlinearity is
$u^2$. The authors observed that the solution of (\ref{QNLS}) with
$Q(u,\bar{u})=u^2$ could be almost entirely supported in the
spacetime-frequency domain $\{(\tau,\xi):\tau>0\}$. Combining this
with some other observations (which we will try to describe below),
they introduced a modified Bourgian space as working space to avoid
the failure in $X_{s,b}$ when $s\leq -\dfrac{3}{4}$. Further, they
showed that the threshold $s\geq -1$ is sharp, that is, (\ref{QNLS})
is ill-posed when $s<-1$, for the nonlinearity $u^2$. Recently, in
\cite{K}, the author showed that (\ref{QNLS}) is well-posedness in
$H^s(\R)$ when $s\geq -1$ and is ill-posed when $s<-1$, for the
nonlinearity $\bar{u}^2$.

In this paper, we are interested in
\begin{equation}
iu_t+u_{xx}=u\bar{u}.\label{NLS}
\end{equation}
We strongly believe that the equation with the nonlinearity
$u\bar{u}$ must behave differently from the two others, as what
presented in \cite{KPV}. One may not expect that the solution in
this case can be almost supported in the region $\{\tau>0\}$. We
believe that the construction of the working space in \cite{BT} is
heavily rely on the nonlinearity $u^2$, and is not well suitable
in this situation. Therefore, we claim that the local result must
be different from \cite{BT}, and we wonder what the differences
are. Indeed, applying the abstract and general theory in
\cite{BT}, we get our first result that the IVP of (\ref{NLS}) is
ill-posed in $H^s(\R)$ when $s<-\dfrac{1}{4}$. That is,
\begin{thm}(Ill-posedness below $H^{-\frac{1}{4}}(\R))$.
The IVP of (\ref{NLS}) is not locally well-posed in $H^s(\R)$ for
any $s<-\dfrac{1}{4}$; more precisely, the solution operator fails
to be uniformly continuous  with respect to the $H^s(\R)$ norm.
\end{thm}

Therefore, we show that the local result related to (\ref{NLS}) in \cite{KPV} is sharp except the endpoint case
when $s=-\dfrac{1}{4}$, which is one of the aim
in this paper.
\begin{thm}
The IVP of (\ref{NLS}) (\ref{1.2}) is locally well-posed for the
initial data $u_0\in H^{s}(\R)$ when $s= -\dfrac{1}{4}$. Moreover, the lifetime
$\delta$ satisfies
$$
\delta\sim \|u_0\|_{H^{s}}^\mu, \makebox{\quad for some\ \ } \mu<0.
$$
\end{thm}

On the other hand, we observe that the ill-posedness of (\ref{NLS})
is caused by the high-high interaction which cascades down into a
very low frequency in the nonlinearity (see the computation in the
proof of Theorem 1.1 in Section 5), while the low-frequency in
$H^s(\R)$ behaves as $L^2$. It implies that one may expect to lower
regularity of the solution in high frequency by working it in
another space which is based on a lower regular space in the low
frequency than $L^2$-norm. For this purpose, we introduce a
modification of Sobolev space $H^s(\R)$. Define
$H^{s,a}(\R)=\{u:\|u\|_{H^{s,a}}<\infty\}$, where
$$
\|u\|_{H^{s,a}} \triangleq\left\|(|\xi|^a\chi_{\{|\xi|\leq
1\}}+|\xi|^s\chi_{\{|\xi|> 1\}})\hat{u}(\xi)\right\|_{L^2_\xi},
$$
where $\chi_A$ is the characteristic function of the set $A$. It's
obvious that Schwartz space is dense in $H^{s,a}(\R)$ when
$a>-\dfrac{1}{2}$. When $a\geq 0$, then $H^s(\R)\hookrightarrow
H^{s,a}(\R)$ (particularly, they are equal when $a=0$). In this paper,
we always restrict that $s\leq 0$ and $a\geq 0$.  We then
turn our attention to study (\ref{NLS}) in
$H^{s,a}(\R)$ and obtain
\begin{thm}(Ill-posedness in $H^{s,a}(\R))$.
The IVP of (\ref{NLS}) is not locally well-posed in $H^{s,a}(\R)$
for any $a>\dfrac{1}{2}$, $s\leq 0$ or $s<
-\dfrac{1}{4}-\dfrac{1}{2}a$, $a\geq 0$; more precisely, the
solution operator fails to be uniformly continuous with respect to
the $H^{s,a}(\R)$ norm.
\end{thm}
\begin{thm}
Let $0\leq a<\dfrac{1}{2}$, $-\dfrac{1}{4}-\dfrac{1}{2}a \leq
s\leq 0$. Then the IVP of (\ref{NLS}) (\ref{1.2}) is locally
well-posed for the initial data $u_0\in H^{s,a}(\R)$. Moreover,
the lifetime $\delta$ satisfies
$$
\delta\sim \|u_0\|_{H^{s,a}}^{\mu'}, \makebox{\quad for
some\quad}\mu'<0.
$$
\end{thm}
In fact, Theorems 1.3 and 1.4 extend the results of Theorems 1.1 and
1.2 respectively, by considering the well-posedness and
ill-posedness theories in the modification Sobolev spaces.

The main technique to prove Theorems 1.2 and 1.4 (together) is  via
a fixed point argument in some modified Bourgain spaces
($\S^{-\rho,a}$, see below). We are indebt in \cite{BT} for the
stimulating arguments. These results do not conclude anything about
the case in $H^{s,a}(\R)$ when $a=\dfrac{1}{2},-\dfrac{1}{2}\leq
s\leq 0$.

 \noindent{\bf Some notations}. We use $A\lesssim B$ or $B\gtrsim A$
to denote the statement that $A\leq CB$ for some large constant $C$
which may vary from line to line.  We use $A\ll
B$ to denote the statement $A\leq C^{-1}B$, and use $A\sim B$ to
mean $A\lesssim B\lesssim A$. The notation $a+$ denotes $a+\epsilon$
for any small $\epsilon$, and $a-$ for $a-\epsilon$.
$\langle\cdot\rangle=(1+|\cdot|^2)^{\frac{1}{2}}$. We use $\|f\|_{L^p_xL^q_t}$
to denote the mixed norm
$\Big(\displaystyle\int\|f(x,\cdot)\|_{L^q}^p\
dx\Big)^{\frac{1}{p}}$. Moreover, we denote
$\hat{u}(\xi)$ and $\tilde{u}(\xi,\tau)$ to be the spatial and  spacetime Fourier transform of
$u$ respectively, and use $\check{f}$ or $\F_{\xi\tau}^{-1}$ to denote the inverse Fourier transform of $f$.

The rest of this article is organized as follows. In Section 2, we construct the
working space. In Section 3, we
derive some preliminary estimates. In Section 4, we recall some general
well-posedness and ill-posedness theories and give the frames of the proof of the
main theorems. In Section 5, we prove Theorems 1.1 and 1.3. In Section
6, we establish the key bilinear estimates to prove Theorems 1.2 and 1.4.

\section{Construction of working space}

In this section, we will construct the working space in building
on Theorems 1.2 and 1.4.
As what implied in \cite{NTT}, the standard Bourgain space
$X_{s,b}$ is not sufficient to handle the well-posedness in the
critical case $H^{-\frac{1}{4}}(\R)$ or some lower regularity
spaces. Moreover, observing the counterexamples in \cite{KPV} and
\cite{NTT}, the failure of $X_{s,b}$ in the bilinear estimate is
caused when the $\tau-\xi^2$ is far away from $\xi^2$ in the
spacetime-frequency domain $(\xi,\tau)$. For this reason, if one
enhances some force in the
 working space to control the behavior of the equation when $\tau-\xi^2$ is large and $\xi$ is small,
then one may avoid those counterexamples. We use the spirit of
\cite{BT} to realize it. For constructing a proper working space,
we need some sum spaces. First, we define some Bourgian-type
spaces by the Fourier transform.

We will take $\hat{X}_{s,b}$ and $\hat{X}^{s,b}$ to be the closure
of the Schwartz functions under the norms$^1$:\footnotetext{$^1$All
sums and unions involving $j$ and $d$ shall be over the nonnegative
unless otherwise mentioned.}
\renewcommand{\arraystretch}{2}
\begin{eqnarray}\label{X_sb}
\|f\|_{\hat{X}_{s,b}}
 &\triangleq&
    \left\|\langle\xi\rangle^s
    \langle\tau-\xi^2\rangle^bf\right\|_{L^2_{\xi\tau}};\\
\|f\|_{\hat{X}^{s,b}}
 &\triangleq&
    \left(\sum_j 2^{2sj}\Big(\sum_d 2^{bd}
    \left\|f\right\|_{L^2_{\xi}L^2_{\tau}(A_j\cap
    B_d)}\Big)^2\right)^{\frac{1}{2}},
    \label{X^sb}
\end{eqnarray}
where
\renewcommand{\arraystretch}{2}
\begin{eqnarray*}
 && A_j\triangleq \{(\xi,\tau)\in \R^2: 2^j\leq \langle\xi\rangle<2^{j+1}\};\\
 && B_d\triangleq \{(\xi,\tau)\in \R^2: 2^d\leq
 \langle\tau-\xi^3\rangle<2^{d+1}\}.
\end{eqnarray*}

\noindent{\it Remark.} $\hat{X}_{s,b}$ are the Fourier transforms
of the standard Bourgain spaces $X_{s,b}$. That is,
$\|u\|_{X_{s,b}}=\|\tilde{u}\|_{\hat{X}_{s,b}}$ for $u\in
X_{s,b}$. Further, we note the relationship that, for any $s\in
\R$, $b'<b$,
\begin{equation}
\hat{X}^{s,b}\hookrightarrow \hat{X}_{s,b}\hookrightarrow
\hat{X}^{s,b'}.
\label{emb1}
\end{equation}
Define the functions $m_{s,a}$ as
$$
m_{s,a}=|\xi|^a\chi_{\{|\xi|\leq 1\}}+|\xi|^s\chi_{\{|\xi|> 1\}}.
$$
Let $X$ and $Y$ are spaces under the norms:
\begin{eqnarray}
\|f\|_{X} & \triangleq &
           \|m_{-\rho,a}f\|_{\hat{X}^{0,\frac{1}{2}}};\\
\|f\|_{Y} & \triangleq &
\|m_{\alpha,a}f\|_{\hat{X}_{0,\beta}}+\|m_{-\rho,a}f\|_{L^2_\xi
L^1_\tau},
\end{eqnarray}
where $0\leq a < \dfrac{1}{2}$, $0\leq \rho\leq \dfrac{1}{4}+\dfrac{1}{2}a$ with
 the parameters $\alpha=\left(\dfrac{1}{4}-\dfrac{1}{2}\rho\right)+,\beta=0+$.

Now, we define our first important space.
\begin{equation}
Z\triangleq X+Y \label{Z}
\end{equation}
with the norm
$$
\|f\|_Z=\inf\Big\{\|f_1\|_X+\|f_2\|_Y: f_1\in X;f_2\in Y;f=f_1+f_2\Big\}.
$$
We give some properties on these spaces.
\begin{lem}$\hat{X}^{0,\frac{1}{2}}\subset L^2_\xi L^1_\tau.$
\end{lem}
{\it Proof}. By a dyadic decomposition on $\xi$, it suffices to show
$$
\|f\|_{L^2_\xi L^1_\tau(A_j)}\lesssim \sum_d
2^{\frac{d}{2}}\|f\|_{L^2_\xi L^2_\tau(A_j\cap B_d)}.
$$
It follows easily from the triangle inequality and H\"{o}lder's
inequality. \hfill$\Box$

Next, we give a pasting lemma between $X$ and $Y$. We define the set
$$
B_{\geq d}\triangleq \bigcup\limits_{d'\geq d}B_{d'};\quad
B_{\leq d}\triangleq \bigcup\limits_{d'\leq d}B_{d'}.
$$
\begin{lem} Let $f$ be a reasonable function and $ k_0\triangleq\dfrac{2\rho+2\alpha}{1-2\beta}$.
\bi
\item[{\rm(1)}]
           If $\mbox{\rm supp}\,f\subset \cup_j(A_j\cap B_{\geq
           k_0j-5})$, then
            $$
            \|f\|_Y= \|f\|_Z;
            $$

\item[{\rm(2)}]
          If $\mbox{\rm supp}\,f\subset \cup_j(A_j\cap B_{\leq
           k_0j+5})$, then
            $$
            \|f\|_X= \|f\|_Z.
            $$
\ei
\end{lem}
{\it Proof}. It's trivial when $j=0$, so we just consider $j\geq 1$.
For (1), we only need to show
$$
\|f\|_{\hat{X}_{\alpha,\beta}}\lesssim \|f\|_{\hat{X}^{-\rho,\frac{1}{2}}},
$$
when $\mbox{\rm supp}\,f\subset \cup_j(A_j\cap B_{\geq k_0j-5})$.
Indeed, we have
\beqs
\begin{split}
\||\xi|^{\alpha}
     \langle\tau-\xi^2\rangle^\beta f\|_{L^2_{\xi\tau}}^2
     &=\sum_j\sum_{d\geq k_0j-5}
       2^{2\alpha j}2^{2\beta d}
       \|f\|_{L^2_{\xi\tau}(A_j\cap B_d)}^2\\
     &\lesssim
       \sum_j\sum_{d\geq k_0j-5}
       2^{(2\rho+2\alpha+k_0(2\beta-1))j}2^{-2\rho j}2^d
       \|f\|_{L^2_{\xi\tau}(A_j\cap B_d)}^2\\
     &=
       \sum_j\sum_{d\geq k_0j-5}
       2^{-2\rho j}2^d
       \|f\|_{L^2_{\xi\tau}(A_j\cap B_d)}^2\\
     &\lesssim
       \sum_j
       2^{-2\rho j}
       \Big(\sum_{d\geq k_0j-5}
       2^{\frac{d}{2}}\|f\|_{L^2_{\xi\tau}(A_j\cap B_d)}\Big)^2.
\end{split}
\eeqs

For (2), it suffices to show that
$$
\|f\|_{\hat{X}^{-\rho,\frac{1}{2}}}\lesssim \|f\|_{\hat{X}_{\alpha,\beta}}.
$$
This follows from
\beqs
\begin{split}
     \sum_{d\leq k_0j+5}2^{\frac{d}{2}}
     \|f\|_{L^2_{\xi\tau}(A_j\cap B_d)}
     &\lesssim
       2^{(\frac{1}{2}-\beta)k_0j}
       \Big(\sum_{d\leq k_0j+5}
       2^{2\beta d}
       \|f\|_{L^2_{\xi\tau}(A_j\cap B_d)}^2\Big)^{\frac{1}{2}}\\
     &\leq
       2^{(\rho+\alpha)j}
       \|\langle\tau-\xi^2\rangle^\beta f\|_{L^2_{\xi\tau}(A_j)},
\end{split}
\eeqs
where we used the Cauchy-Schwartz's inequality in the first step.
\hfill$\Box$

We are ready to define our working space. Let
$\S^{-\rho,a},\N^{-\rho,a}$
be the closure of the Schwartz functions under the norms
\renewcommand{\arraystretch}{2}
\begin{eqnarray}\label{S}
 && \|u\|_{\S^{-\rho,a}}
    \triangleq\|\tilde{u}\|_{Z};\\
 && \|u\|_{\N^{-\rho,a}}
    \triangleq\left\|\frac{\tilde{u}(\xi,\tau)}{\langle\tau-\xi^2\rangle}\right\|_Z,
    \label{N}
\end{eqnarray}
where we write $s=-\rho$, and $\S^{-\rho,a}$ is our working space,
$\N^{-\rho,a}$ is the space related to Duhamel term.
 It's easy to see that they are both Banach spaces.

\noindent{\it Remark.} The space $X^{s,b}$ is a stronger spaces than
$X_{s,b}$ and can be regarded as a refined space of $X_{s,b}$ in
some situations. However, the space $X^{s,b}$ is seemly still
insufficient to handle the critical case $s=-\dfrac{1}{4}$ in
$H^s(\R)$, because of the weight of $l^1$-norm in (\ref{X^sb}). For
this reason, we shall add the weaker space (for the same exponents)
$\hat{X}_{s,b}$ with $b<\dfrac{1}{2}$ to deal with the high-to-very
low ($\{|\xi|\leq 1\}$) frequency cascade case.

\section{Some Preliminary Estimates}

We will denote by $\{S(t)\}_{t\in \R}$ to be the unitary group
generated by the corresponding linear equation of (\ref{NLS}) \beq
 v_t+v_{xx}=0, \qquad x,\,t\in \R,
 \label{LS}
\eeq such that $v=S(t)u_0$ solves (\ref{LS})(\ref{1.2}). It is also
defined explicitly by spatial Fourier transform as
$$
\widehat{S(t)u_0}(\xi)\triangleq e^{-it\xi^2}\widehat{u_0}(\xi).
$$

First, we present a well-known Stricharz estimate due to Bourgain
space (see \cite{G} for example). Recall that $X_{s,b} $ is the
standard Bourgain space, then
\begin{lem}
For $u\in X_{0,\frac{1}{2}+}$, we have
\begin{equation}
\|u\|_{L^6_{xt}}
    \lesssim
    \|u\|_{X_{0,\frac{1}{2}+}}.\label{XE}
\end{equation}
\end{lem}
By interpolating between (\ref{XE}) and the following equality
\beq
 \|F\|_{L^2_ {xt}}=
 \|F\|_{X_{0,0}},\label{L2}
\eeq
we can generalize (\ref{XE}) as below.
\begin{lem}For
$\theta\geq  \dfrac{3}{2}\left(\dfrac{1}{2}-\dfrac{1}{q}\right)$, $q\in [2,6]$ and
$F\in X_{0,\theta+}$,  we have
\beq
 \|F\|_{L^q_{xt}}\lesssim \|F\|_{X_{0,\theta+}}.\label{XE1}
\eeq
\end{lem}

Next, we introduce some multiplier operators (appeared in \cite{G},
but with another versions). For nonnegative functions $f,g,h$,
define
\begin{equation}
I^s_k(f,g,h)=\displaystyle\int_\ast m_k(\xi,\xi_1,\xi_2)^s f(\xi_1,\tau_1)g(-\xi_2,-\tau_2)h(\xi,\tau)
\label{Isk}
\end{equation}
for  $k=1,2,3$, where $\displaystyle\int_\ast
=\int_{\stackrel{\xi_1+\xi_2=\xi,}{
\tau_1+\tau_2=\tau}}\,d\xi_1 d\xi_2 d\tau_1 d\tau_2$, and the multipliers $m_k$ are defined as
$$
m_1=|\xi|; \quad m_2=|\xi+\xi_2|; \quad m_3=|\xi_2|.
$$
Then we have
\begin{lem}
Let $f,g,h$ are reasonable functions, then
\begin{eqnarray}
I^{\frac{1}{2}}_1(f,g,h)
  & \lesssim &
              \|h\|_{L^2}\,
              \|f\|_{\hat{X}_{0,\frac{1}{2}+}}
              \|g\|_{\hat{X}_{0,\frac{1}{2}+}};\label{36}\\
I^{\frac{1}{2}}_2(f,g,h)
  & \lesssim &
              \|f\|_{L^2}\,
              \|g\|_{\hat{X}_{0,\frac{1}{2}+}}
              \|h\|_{\hat{X}_{0,\frac{1}{2}+}}; \nonumber\\
I^{\frac{1}{2}}_3(f,g,h)
  & \lesssim &
              \|g\|_{L^2}\,
              \|f\|_{\hat{X}_{0,\frac{1}{2}+}}
              \|h\|_{\hat{X}_{0,\frac{1}{2}+}}.\label{3.6}
\end{eqnarray}
\end{lem}
\begin{proof}
We use the argument in \cite{CKS1} to prove the lemma. For $I^{\frac{1}{2}}_1$, we change variables by setting
\begin{equation}
\tau=\lambda+\xi^2,\quad \tau_1=\lambda_1+\xi_1^2,\quad \tau_2=\lambda_2-\xi_2^2,\label{lambda}
\end{equation}
then,  $I^{\frac{1}{2}}_1(f,g,h)$ is changed into
\begin{equation}
\displaystyle\int m_1^{\frac{1}{2}}\, f(\xi_1,\lambda_1+\xi_1^2)\,
g(-\xi_2,-\lambda_2+\xi_2^2)\,
h(\xi_1+\xi_2,\lambda_1+\lambda_2+\xi_1^2-\xi_2^2)\, d \xi_1 d \xi_2
d \lambda_1 d \lambda_2.\label{2.9}
\end{equation}
We change variables again as follows. Let
\begin{equation}
 (\eta,\omega)=T(\xi_1,\xi_2),\label{transform}
\end{equation}
where
\beqs
\begin{split}
 \eta & = T_1(\xi_1,\xi_2)=\xi_1+\xi_2,\\
 \omega & = T_2(\xi_1,\xi_2)=\lambda_1+\lambda_2+\xi_1^2-\xi_2^2.
\end{split}
\eeqs
Then the Jacobian $J$ of this transform satisfies
$$
|J|=2|\xi_1+\xi_2|.
$$
Define
$$
H(\eta,\omega,\lambda_1,\lambda_2)=
fg\circ
T^{-1}(\eta,\omega,\lambda_1,\lambda_2),
$$
then, by eliminating $|J|^{\frac{1}{2}}$ with
$m_1^{\frac{1}{2}}$, (\ref{2.9}) has a bound of
\begin{equation}
    \displaystyle\int h(\eta,\omega)
    \cdot \dfrac{ H(\eta,\omega,\lambda_1,\lambda_2)}
    {|J|^{\frac{1}{2}}}
    \,d\eta d\omega d\lambda_1 d\lambda_2.\label{2.11}
\end{equation}
By H\"{o}lder' inequality, we have
\renewcommand{\arraystretch}{2}
\beqs
\begin{split}
(\ref{2.11})\leq
 &\  \|h\|_{L^2_{\eta\omega}}
     \cdot\displaystyle\int
      \Big(\int \dfrac{|H(\eta,\omega,\lambda_1,\lambda_2)|^2}{|J|}
      \,d\eta \omega\Big)^{\frac{1}{2}}
      \,d\lambda_1 d\lambda_2\\
\lesssim
 &\   \|h\|_{L^2}
      \displaystyle\int
      \left\|f(\xi_1,\lambda_1+\xi_1^2)\right\|_{L^2_{\xi_1}}\,d\lambda_1
      \cdot
      \displaystyle\int
      \left\|g(-\xi_2,-\lambda_2+\xi_2^2)\right\|_{L^2_{\xi_2}}\,d\lambda_2\\
\lesssim
 &\   \|h\|_{L^2}
      \|f\|_{\hat{X}_{0,\frac{1}{2}+}}
      \|g\|_{\hat{X}_{0,\frac{1}{2}+}},
\end{split}
\eeqs
where we employed the inverse transform of (\ref{transform}) in the
second step and H\"{o}lder' inequality in the third step.

For $I^{\frac{1}{2}}_2$, the modification of the proof is replacing the variable transform $(\eta,\omega)$ by
\beqs
\begin{split}
 \eta & = T_1(\xi,\xi_2)=\xi-\xi_2,\\
 \omega & = T_2(\xi,\xi_2)=\lambda-\lambda_2+\xi^2+\xi_2^2.
\end{split}
\eeqs
Then the Jacobian $J$ in this situation satisfies
$$
|J|=2|\xi+\xi_2|.
$$
Therefore, we have the claim by the same argument as above.

For $I^{\frac{1}{2}}_3$, we take
\beqs
\begin{split}
 \eta & = T_1(\xi,\xi_1)=-\xi+\xi_1,\\
 \omega & = T_2(\xi,\xi_1)=-\lambda+\lambda_1-\xi^2+\xi_1^2.
\end{split}
\eeqs
in this time. Then the Jacobian $J$ in this situation satisfies
$$
|J|=2|\xi_2|.
$$
So the claim follows again. \hfill$\Box$
\end{proof}

When $s=0$, by (\ref{XE1}) we have
\begin{eqnarray}
I^{0}_1(f,g,h)
 &\leq     & \|h\|_{L^2}\, \|\check{f}\|_{L^p}\,\|\check{g}\|_{L^q}\nonumber\\
 &\lesssim & \|h\|_{L^2}\, \|f\|_{\hat{X}_{0,b+}}\,\|g\|_{\hat{X}_{0,b'+}},
\label{2.10}
\end{eqnarray}
where
$
\dfrac{1}{p}+\dfrac{1}{q}=\dfrac{1}{2}, b=\dfrac{3}{2}\Big(\dfrac{1}{2}-\dfrac{1}{p}\Big)$,
$  b'=\dfrac{3}{2}\Big(\dfrac{1}{2}-\dfrac{1}{q}\Big),$
that is $b+b'=\dfrac{3}{4}$, and $b,b'\in \big[\dfrac{1}{4},\dfrac{1}{2}\big]$.

Interpolation between (\ref{36}) and (\ref{2.10}) twice, we have
\begin{cor}
Let $I^{s}_1$ be defined by (\ref{Isk}), then for any $s\in
[0,\dfrac{1}{2}]$,
\begin{equation}
I^{s}_1(f,g,h)
\lesssim
\|h\|_{L^2}\, \|f\|_{\hat{X}_{0,b_1+}}\,\|g\|_{\hat{X}_{0,b_2+}},
\label{IE1}
\end{equation}
where $b_1=\dfrac{1}{2}(1-s'+s)$, $b_2=\dfrac{1}{4}(2s'+1)$ for any
$s'\in [s,\dfrac{1}{2}]$.
\end{cor}

For $I^{s}_2$ and  $I^{s}_3$, the similar estimates hold too.
But in this paper, we just need the following crude estimates. For any $s<\dfrac{1}{2}$,
\begin{eqnarray}
I^{s}_1(f,g,h)
 &\lesssim &
  \|h\|_{L^2}\, \|f\|_{\hat{X}_{0,\frac{1}{2}-}}\,\|g\|_{\hat{X}_{0,\frac{1}{2}-}},\label{I1}\\
I^{s}_2(f,g,h)
 &\lesssim &
  \|f\|_{L^2}\, \|g\|_{\hat{X}_{0,\frac{1}{2}-}}\,\|h\|_{\hat{X}_{0,\frac{1}{2}-}},\label{I2}\\
I^{s}_3(f,g,h)
 &\lesssim &
  \|g\|_{L^2}\, \|f\|_{\hat{X}_{0,\frac{1}{2}-}}\,\|h\|_{\hat{X}_{0,\frac{1}{2}-}}.\label{I3}
\end{eqnarray}

\noindent{\it Remark.} Sometimes, one may interested in some
critical estimates in Lemma 3.2 by replace $b=\dfrac{1}{2}+$ by
$b=\dfrac{1}{2}$, which may be useful to deal with some limiting
case. In general, one may not have  such critical estimates in
$\hat{X}^{s,b}$. However, in some especial case, for example, when
$|\xi|\ll |\xi_1|$ in $I^{s}_1(f,g,h)$ (similar for
$I^{s}_2,I^{s}_3$), they hold in  $\hat{X}^{s,b}$. Particularly, if
we consider another Bourgain spaces $\hat{\underline{X}}^{s,b}$,
defined by the norm
$$
\|f\|_{\hat{\underline{X}}^{s,b}}
    \triangleq
    \sum_d 2^{bd}\|\langle\xi\rangle^sf\|_{L^2_{\xi\tau}(B_d)},
$$
then the critical estimates hold in these spaces. In fact, the spaces $\hat{\underline{X}}^{s,b}$ are
stronger than $\hat{X}^{s,b}$ in the sense  that
\begin{equation}
\|f\|_{\hat{X}^{s,b}}\leq \|f\|_{\hat{\underline{X}}^{s,b}}, \label{emb2}
\end{equation}
which easily follows by the triangle inequality of $l^2$-norm.
\vspace{0.5cm}
\section{Preparatory Theory}

Recall the scale invariance that, if $u(x,t)$ is a solution of
IVP (\ref{NLS}) (\ref{1.2}), then for any $\lambda>0$,
\begin{eqnarray}
u_\lambda(x,t)=\lambda^{-2}u(x/\lambda,t/\lambda^2)
\label{scale}
\end{eqnarray}
is also a solution of (\ref{NLS}) with the initial data replaced by
$$
  u_{0,\lambda}(x)=\lambda^{-2}u_0(x/\lambda).
$$
Note that for $\lambda>1$,
$$
\|u_{0,\lambda}\|_{H^{s,a}}\lesssim \lambda^{-\frac{3}{2}-s}\|u_0\|_{H^{s,a}},
\makebox{\quad when } a\geq s.
$$
So we can scale the initial data to be a small size in $H^{s,a}(\R)$ ($H^{s}(\R)$ when $a=0$) for $s>-\dfrac{3}{2}$.
It thus suffices to prove Theorems 1.2 and 1.4 for small initial data in $H^{-\frac{1}{4}}(\R)$
and $H^{s,a}(\R)$ respectively, with the lifetime $\delta=1$.

Next, by the Duhamel's formula, we can rewrite (\ref{NLS}) (\ref{1.2}) in integral form as
$$
u(t)=S(t)u_0+\displaystyle\int_0^t S(t-t')u(t')\bar{u}(t')\,dt'.
$$
If we are interested in (locally) solving the IVP up to time $\delta=1$, then it can be replaced by
\begin{eqnarray}
u(t)&=&\eta(t)S(t)u_0+\eta(t)\displaystyle\int_0^t S(t-t')u(t')\bar{u}(t')\,dt'\nonumber\\
&\triangleq &L(u_0)+N(u,u),\label{IEQ}
\end{eqnarray}
where
$$
L(u_0)\triangleq \eta(t)S(t)u_0;\quad N(u,v)\triangleq
\eta(t)\displaystyle\int_0^t S(t-t')u(t')\bar{v}(t')\,dt',
$$
and $\eta(t)$ is a smooth bump function supported in the interval
$[-2,2]$ such that $\eta(t)=1$ on $[-1,1]$.

Now we recall some well-posedness and ill-posedness theories
established in \cite{BT} (a little general in this paper). We shall
be somewhat brief here and refer the reader to \cite{BT} for more
details. Let $(D,\|\cdot\|_D)$ and $(S,\|\cdot\|_S)$ are Banach
spaces, such that $L: D\rightarrow S$ and $N: S\times S\rightarrow
S$ are densely defined. If \bi
\item[{\rm(1)}] \, \, $\|L(u_0)\|_S\lesssim \|u_0\|_D$;

\item[{\rm(2)}] \,\,  $\|N(u,v)\|_S\lesssim \|u\|_S\|v\|_S$.
\ei
Then we say the equation (\ref{IEQ}) is quantitatively well-posed in $D,S$.
By a standard fixed point argument, if (\ref{IEQ}) is quantitatively well-posed in $D,S$, then it has local
existence, continuity, uniqueness in $S$, when $\|u_0\|_D$ is small enough. If one still has
the energy estimate
\bi
\item[{\rm(3)}] \, \, $\|u\|_{C^0_t([0,1];D)}\lesssim \|u_0\|_S$,
\ei
then  (\ref{IEQ}) is locally well-posed for the data in $D$ with small norm.

Concretely, we consider $D$ to be a weighted $L^2$ space. Define
$$
\|u_0\|_D\triangleq\|m(\xi)\hat{u}(\xi)\|_{L^2_\xi};
\quad \|u\|_S\triangleq \|\tilde{u}\|_{\hat{S}}\triangleq  \|m(\xi)\tilde{u}\|_{\hat{S}^0}
$$
for some function $m$ and Hilbert space $\hat{S}^0$. Since
\begin{equation}
\widetilde{\bar{u}}(\xi,\tau)=\bar{\tilde{u}}(-\xi,-\tau),\label{4.3}
\end{equation}
 by the argument in \cite{BT} (with a bit modification), (1)--(3) can be replaced by
\bi
\item[{\rm(i)}] \, \, $|f|\leq |g|$, then $\|f\|_{\hat{S}}\leq \|g\|_{\hat{S}}$;

\item[{\rm(ii)}] \,\,  $\|f\|_{L^2_\xi L^1_\tau}\lesssim \|f\|_{\hat{S}^0}$;

\item[{\rm(iii)}] \,\,  $\|f\|_{\hat{S}^0}\lesssim \|f\|_{\hat{X}_{0,100}}$;

\item[{\rm(iv)}] \,\,  $\|\langle\tau-\xi^2\rangle^{-1} (f\ast g^\star)\|_{\hat{S}}
                        \lesssim \|f\|_{\hat{S}}\,\|g\|_{\hat{S}}$, where $g^\star(\xi,\tau)=g(-\xi,-\tau)$.
\ei
It easy to see that  $(D, S)=(H^{-\rho,a}, \S^{-\rho,a})$ satisfies
the condition (i)--(iii). So we can see that the only work left to finish the proof of Theorems 1.2 and 1.4 (together)
is the bilinear estimate (iv) in the solution space. More precisely, to prove Theorems 1.2 and 1.4,
(iv) is equivalent to
\begin{equation}
\left\|\dfrac{1}{\langle\tau-\xi^2\rangle} (\tilde{u}\ast \tilde{v}^\star)\right\|_{Z}\lesssim
\|\tilde{u}\|_{Z}\,\|\tilde{v}\|_{Z},\label{BE}
\end{equation}
where $\tilde{v}^\star(\xi,\tau)= \tilde{v}(-\xi,-\tau)$. It will be established in Section 6.

Next, suppose that (\ref{IEQ}) is  quantitatively well-posed in $D,S$. If we define the nonlinear map
$A_n: D\rightarrow S$ for $n=1,2,\cdots$ as
\begin{eqnarray*}
A_1(u_0)
   & \triangleq & L(u_0);\\
A_n(u_0)
   & \triangleq & \sum_{n_1,n_2\geq1;n_1+n_2=n }
   N(A_{n_1}(u_0), A_{n_2}(u_0)) \mbox{\quad for } n> 1,
\end{eqnarray*}
then the solution map
$$
u[u_0]=\sum_{n=1}^\infty A_n(u_0) \mbox{\quad in } S \mbox{ (absolutely convergent)}
$$
for small data $u_0\in D$. Moreover, we have
\begin{prop}(\cite{BT}) Suppose that (\ref{IEQ}) is  quantitatively well-posed in $D,S$, with
a solution map $u_0\mapsto u[u_0]$ from a ball $B_D$ in $D$ to a ball  $B_S$ in $S$. Suppose that
these spaces are then given other norms $D'$ and $S'$, which are weaker than  $D$ and $S$ in the sense
that
$$
\|u_0\|_{D'}\lesssim \|u_0\|_{D},\quad \|u\|_{S'}\lesssim \|u\|_{S}.
$$
Suppose that the solution map $u_0\mapsto u[u_0]$ is continuous  from $(B_D,\|\cdot\|_{D'})$
to $(B_S,\|\cdot\|_{S'})$. Then for each $n$, the nonlinear operator $A_n: D\rightarrow S$
 is continuous  from $(B_D,\|\cdot\|_{D'})$
to $(B_S,\|\cdot\|_{S'})$.
\end{prop}
This proposition gives us a way to disprove well-posedness in coarse topologies, simply by establishing
that at least one of the operators $A_n$ is discontinuous.

\vspace{0.5cm}
\section{Some Ill-posedness Analysis}

In this section, we concentrate our attention on the consequences which are derived from
the application on Proposition 4.1. We expect to obtain some necessary restriction on
the regularity exponents for well-posedness theory. Roughly speaking, by Proposition 4.1,
if the solution map is continuous from $D$ to $S$, then so is the quadratic
$$
A_2: u_0\mapsto N(Lu_0,Lu_0).
$$

We consider $D=H^{s,a}(\R)$, $S=C_t^0([0,1];H^{s,a}(\R))$, in which we set
the lifetime $\delta=1$ by scale invariance. Fix $N\gg 1$ and $\varepsilon_0 \ll 1$, set
$$
\widehat{u_0}(\xi)=\varepsilon_0 N^{-s} \chi_{[-10,10]}(|\xi|-N),
$$
then $\|u_0\|_{H^{s,a}}\sim \varepsilon_0$, for any $a\in \R^+$.

First, we consider the case  $D=H^{s}(\R)$ and $S=C_t^0([0,1];H^{s}(\R))$, then
 $\|A_2\|_{C_t^0([0,1];H^{s})}$ is equal to
\begin{eqnarray}
   &   & \sup_{0\leq t \leq 1}\left\|\displaystyle\int_0^t S(t-t')
         \left(S(t')u_0\cdot
         \overline{S(t')u_0}\right)\,dt'\right\|_{H^{s}}\nonumber\\
   & = & \sup_{0\leq t \leq 1}\biggl\|\langle\xi\rangle^{s}\displaystyle\int_0^t\!\!\!\int \exp(-i(t-t')
         \xi^2)[\exp(-it'\xi_1^2)\widehat{u_0}(\xi_1)\nonumber\\
   &   &
         \cdot
         \exp(it'(\xi-\xi_1)^2)\widehat{u_0}(-\xi+\xi_1)]\,d\xi_1 dt'\biggl\|_{L^2_{\xi}}\nonumber\\
   & = & \sup_{0\leq t \leq 1}\biggl\|\langle\xi\rangle^{s}\displaystyle\int_0^t\!\!\!\int \exp(-it
         \xi^2)\exp(2it'\xi(\xi-\xi_1))\widehat{u_0}(\xi_1)\widehat{u_0}(-\xi+\xi_1)\,d\xi_1 dt'\biggl\|_{L^2_{\xi}}
         \label{5.1}
\end{eqnarray}
by the Fourier transform and (\ref{4.3}) in the second step. Further, (\ref{5.1}) has a lower bound of
\begin{equation}
\sup_{0\leq t \leq 1}\biggl\|\displaystyle\int_0^t\!\!\!\int \exp(-it
         \xi^2)\exp(2it'\xi(\xi-\xi_1))\widehat{u_0}(\xi_1)\widehat{u_0}(-\xi+\xi_1)
         \,d\xi_1 dt'\biggl\|_{L^2_{\xi}(\frac{1}{100N},\frac{1}{10N})}.
         \label{5.2}
\end{equation}
Note that, for $\xi\in \Big[\dfrac{1}{100N},\dfrac{1}{10N}\Big]$,
\begin{equation}
\mbox{Re}\left(\exp(-it\xi^2)\exp(2it'\xi(\xi-\xi_1))\right)>\dfrac{1}{2},\label{5.3}
\end{equation}
whenever $0\leq t'\leq t \leq 1$ and $\xi_1$ resides in the support of $u_0$. Hence, we have
\begin{eqnarray*}
(\ref{5.2})
   &\gtrsim & N^{-2s}  \,\|1\|_{L^2_{\xi}(\frac{1}{100N},\frac{1}{10N})}\\
   & \sim & N^{-2s-\frac{1}{2}}.
\end{eqnarray*}
For the continuity of $A_2$, it's necessary that
 $s\geq -\dfrac{1}{4}$. This proves Theorem 1.1.

From the computation above, the threshold is much restricted by the
$L^2-$norm in the low frequency in $H^s(\R)$. It's a reason that we
consider the  modification spaces $H^{s,a}(\R)$ to lower the
regularity in low frequency. A similar computation (but replaces
$H^{s}(\R)$ by $H^{s,a}(\R)$) shows that the necessary condition on
$s$ is changed into
\begin{equation}
s\geq -\frac{1}{4}-\frac{1}{2}a. \label{sa}
\end{equation}
One may thus expert to lower the exponent $s$ by setting $a>0$.

On the other hand, the exponent $s$ can't lower to $-\infty$ by choosing various $a$ in (\ref{sa}).
Indeed, if we localize $\xi$ to the region $(1,2)$, then similarly,
$\|A_2\|_{C_t^0([0,1];H^{s,a})}$ has a lower bound of
\begin{equation}
\sup_{0\leq t \leq 1}\biggl\|\displaystyle\int_0^t\!\!\!\int \exp\left(-it
         \xi^2\right)\exp(2it'\xi(\xi-\xi_1))\widehat{u_0}(\xi_1)\widehat{u_0}(-\xi+\xi_1)
         \,d\xi_1 dt'\biggl\|_{L^2_{\xi}(1,2)}.\label{5.4}
\end{equation}
Set $t=\dfrac{1}{100}N^{-1}$ now, then again we have (\ref{5.3}), and (\ref{5.4})
has a lower bound of
$
N^{-2s-1},
$
which implies another restriction that $s\geq -\dfrac{1}{2}$ for each $a\in \R^+$.

Moreover, set a new data
$$
\widehat{u_0}(\xi)=\varepsilon_0 N^{a+\frac{1}{2}} \chi_{[N^{-1},2N^{-1}]}(|\xi|),
$$
then $\|u_0\|_{H^{s,a}}\sim \varepsilon_0$, for any $s\in \R$. On the other hand, $\|A_2\|_{C_t^0([0,1];H^{s,a})}$
is equal to
\beqs
\begin{split}
     &   \,\,\sup_{0\leq t \leq 1}\biggl\||\xi|^a \displaystyle\int_0^t\!\!\!\int \exp(-it\xi^2)
         \exp(2it'\xi(\xi-\xi_1))\widehat{u_0}(\xi_1)\widehat{u_0}(-\xi+\xi_1)\,d\xi_1 dt'\biggl\|_{L^2_{\xi}}\\
\geq &
         \,\,\sup_{0\leq t \leq 1}\biggl\||\xi|^a\displaystyle\int_0^t\!\!\!\int
         \mbox{Re}\left(\exp(-it\xi^2\right)
         \exp\left(2it'\xi(\xi-\xi_1)\right)
         \widehat{u_0}(\xi_1)
         \widehat{u_0}(-\xi+\xi_1)\,d\xi_1 dt'\biggl\|_{L^2_{\xi}(\frac{1}{100N},\frac{1}{10N})}\\
\gtrsim &
         \,\,N^{2(a+\frac{1}{2})}\,N^{-1}\,N^{-a}\,N^{-\frac{1}{2}}\\
   =    &
         \,\,N^{a-\frac{1}{2}},
\end{split}
\eeqs
which implies  the necessary condition on the exponent $a$ of $a\leq \dfrac{1}{2}$ for each $s\in \R$. Thus proves
Theorem 1.3.

\vspace{0.5cm}
\section{Bilinear Estimates}
As discussing above, in order to prove Theorem 1.4, we just need
(\ref{BE}). By the pasting Lemma 2.3, we divide the proof of
(\ref{BE}) into four cases. It will be very convenient to using the
estimate (\ref{emb2}) in the following precess.
\begin{lem}
When $\mbox{\rm supp}\,\tilde{u},\tilde{v}
\subset \cup_j(A_j\cap B_{\geq k_0j})$, then
\begin{equation}
\left\|\dfrac{1}{\langle\tau-\xi^2\rangle}
(\tilde{u}\ast \tilde{v}^\star)\right\|_{X}
\lesssim
\|\tilde{u}\|_{Y}\,\|\tilde{v}\|_{Y}.\label{BE1}
\end{equation}
\end{lem}
\noindent{\it Proof.} By (\ref{emb2}), it suffices to show that
\begin{equation*}
\sum_d 2^{-\frac{d}{2}}
\|m_{-\rho,a}(\xi)(\tilde{u}\ast\tilde{v}^\star)\|_{L^2_{\xi\tau}(B_d)}
\lesssim
\|m_{\alpha,a}\tilde{u}\|_{X_{0,\beta}}\,\|m_{\alpha,a}\tilde{v}\|_{X_{0,\beta}},
\end{equation*}
which is equivalent to show
\begin{eqnarray}
&&\sum_d 2^{-\frac{d}{2}}\left\|m_{-\rho,a}(\xi)
\displaystyle \int_\star \dfrac{f(\xi_1,\tau_1)}{m_{\alpha,a}(\xi_1)\langle\tau_1-\xi_1^2\rangle^\beta}
\dfrac{g(-\xi_2,-\tau_2)}{m_{\alpha,a}(\xi_2)\langle\tau_2+\xi_2^2\rangle^\beta}\right\|_{L^2_{\xi\tau}(B_d)}
\nonumber\\
&&\lesssim
\|f\|_{L^2_{\xi\tau}}\,\|g\|_{L^2_{\xi\tau}}\label{6.2}
\end{eqnarray}
for any $f,g\in L^2(\R^2)$, where
$\displaystyle\int_\star
=\int_{\stackrel{\xi_1+\xi_2=\xi,}{\tau_1+\tau_2=\tau}}\,d\xi_1 d\tau_1$.
We may only consider the integration over the region of
$|\xi_1|\geq|\xi_2|$ (it's similar for $|\xi_1|\leq|\xi_2|$). Then we divide (\ref{6.2}) into
two parts to analyze.
$$
\mbox{Part 1. } |\xi_2|\lesssim 1;\quad
\mbox{Part 2. } |\xi_2|\gg 1.
$$

\noindent {\bf Part 1}. $|\xi_2|\lesssim 1$. Note that $|\xi|\lesssim |\xi_1|$, so we always have
$$
m_{-\rho,a}(\xi)\cdot m_{\alpha,a}(\xi_1)^{-1}\lesssim 1,
$$
no matter when $|\xi_1|\leq 1$ or  $|\xi_1|\geq 1$. Therefore, we
have
\begin{eqnarray*}
&& \left\|m_{-\rho,a}(\xi)
   \displaystyle \int_\star \dfrac{f(\xi_1,\tau_1)}{m_{\alpha,a}(\xi_1)\langle\tau_1-\xi_1^2\rangle^\beta}
   \dfrac{g(-\xi_2,-\tau_2)}
   {m_{\alpha,a}(\xi_2)\langle\tau_2+\xi_2^2\rangle^\beta}\right\|_{L^2_{\xi\tau}(B_d)}\\
&\lesssim&
   \left\| \dfrac{f}{\langle\tau-\xi^2\rangle^\beta}\ast
   \dfrac{|\xi|^{-a} g^\star}{\langle\tau+\xi^2\rangle^\beta}\right\|_{L^2_{\xi\tau}( B_d)}\\
&\lesssim&
   \|1\|_{L^\infty_\xi L^q_\tau(B_d)}\,\left\| \dfrac{f}{\langle\tau-\xi^2\rangle^\beta}\ast
   \dfrac{|\xi|^{-a} g^\star}{\langle\tau+\xi^2\rangle^\beta}\right\|_{L^2_{\xi} L^{q_1}_\tau}\\
&\lesssim&
   2^{d/q}\,
   \left\|\dfrac{f}{\langle\tau-\xi^2\rangle^\beta}
   \right\|_{L^2_{\xi} L^{q_2}_\tau}\,
   \left\|\dfrac{|\xi|^{-a} g^\star}{\langle\tau+\xi^2\rangle^\beta}
   \right\|_{L^1_{\xi} L^{q_3}_\tau(|\xi|\lesssim 1)}\\
&\lesssim&
   2^{d/q}\,
   \|f\|_{L^2_{\xi\tau}}\,\|\langle\tau-\xi^2\rangle^{-\beta}\|_{L^\infty_{\xi} L^{q_4}_\tau}\,
   \|g^\star\|_{L^2_{\xi\tau}}\,
   \||\xi|^{-a}\langle\tau+\xi^2\rangle^{-\beta}\|_{L^2_{\xi} L^{q_5}_\tau(|\xi|\lesssim 1)}\,\\
&\lesssim&
   2^{d/q}\,\|f\|_{L^2_{\xi\tau}}\,\|g\|_{L^2_{\xi\tau}},
\end{eqnarray*}
where
$$
\dfrac{1}{q}+\dfrac{1}{q_1}=\dfrac{1}{2};\,\, \dfrac{1}{q_1}=\dfrac{1}{q_2}+\dfrac{1}{q_3}-1;\,\,
\dfrac{1}{q_2}=\dfrac{1}{2}+\dfrac{1}{q_4};\,\,\dfrac{1}{q_3}=\dfrac{1}{2}+\dfrac{1}{q_5}
$$
with $q>2$, $\beta q_4>1$, $\beta q_5>1$, $a<\dfrac{1}{2}$. By an
elementary computation, we see that $q,q_i,i=1,\cdots,5$ are
reasonable when $\beta>0, a<\dfrac{1}{2}$.

\noindent {\bf Part 2}. $|\xi_2|\gg 1$. Then $|\xi_1|\gg 1$, and we have
\begin{eqnarray*}
&& \left\|m_{-\rho,a}(\xi)
   \displaystyle \int_\star \dfrac{f(\xi_1,\tau_1)}
   {m_{\alpha,a}(\xi_1)\langle\tau_1-\xi_1^2\rangle^\beta}
   \dfrac{g(-\xi_2,-\tau_2)}
   {m_{\alpha,a}(\xi_2)\langle\tau_2+\xi_2^2\rangle^\beta}
   \right\|_{L^2_{\xi\tau}(B_d)}\\
&=&
   \left\|m_{-\rho,a}(\xi) \left(\dfrac{f}{|\xi|^\alpha\langle\tau-\xi^2\rangle^\beta}\ast
   \dfrac{g^\star}{|\xi|^\alpha\langle\tau+\xi^2\rangle^\beta}\right)
   \right\|_{L^2_{\xi\tau}(B_d)}\\
&\lesssim&
   \|m_{-\rho,a}(\xi)\|_{L^{p}_\xi L^q_\tau(B_d)}\,
   \left\| \dfrac{f}{|\xi|^\alpha\langle\tau-\xi^2\rangle^\beta}\ast
   \dfrac{g^\star}{|\xi|^\alpha\langle\tau+\xi^2\rangle^\beta}\right\|_{L^{p_1}_{\xi} L^{q_1}_\tau}\\
&\lesssim&
   2^{d/q}\,
   \left\|\dfrac{f}{|\xi|^\alpha\langle\tau-\xi^2\rangle^\beta}
   \right\|_{L^{p_2}_{\xi} L^{q_2}_\tau(|\xi|\gg 1)}\,
   \left\|\dfrac{ g^\star}{|\xi|^\alpha\langle\tau+\xi^2\rangle^\beta}
   \right\|_{L^{p_3}_{\xi} L^{q_3}_\tau(|\xi|\gg 1)}\\
&\lesssim&
   2^{d/q}\,
   \|f\|_{L^2_{\xi\tau}}\,
   \||\xi|^{-\alpha}\langle\tau-\xi^2\rangle^{-\beta}\|_{L^{p_4}_{\xi} L^{q_4}_\tau(|\xi|\gg 1)}\,\\
&        &
   \|g^\star\|_{L^2_{\xi\tau}}\,
   \||\xi|^{-\alpha}\langle\tau+\xi^2\rangle^{-\beta}\|_{L^{p_5}_{\xi} L^{q_5}_\tau(|\xi|\gg 1)}\,\\
&\lesssim&
   2^{d/q}\,\|f\|_{L^2_{\xi\tau}}\,\|g\|_{L^2_{\xi\tau}},
\end{eqnarray*}
where $q,q_i,i=1,\cdots,5$ as Part 1, and
$$
\dfrac{1}{p}+\dfrac{1}{p_1}=\dfrac{1}{2};\,\, \dfrac{1}{p_1}=\dfrac{1}{p_2}+\dfrac{1}{p_3}-1;\,\,
\dfrac{1}{p_2}=\dfrac{1}{2}+\dfrac{1}{p_4};\,\,\dfrac{1}{p_3}=\dfrac{1}{2}+\dfrac{1}{p_5}
$$
with  $\rho p>1$, $\alpha p_4>1$, $\alpha p_5>1$. They are reasonable when
$$
2\alpha>\dfrac{1}{2}-\rho, \,\beta>0.
$$
This completes the proof of the lemma. \hfill$\Box$

\begin{lem}
When
$\mbox{\rm supp}\,\tilde{u}
\subset \cup_j(A_j\cap B_{\leq k_0j})$,
$\mbox{\rm supp}\,\tilde{v}
\subset \cup_j(A_j\cap B_{\geq k_0j})$, then
\begin{equation}
\left\|\dfrac{1}{\langle\tau-\xi^2\rangle} (\tilde{u}\ast
\tilde{v}^\star)\right\|_{X} \lesssim
\|\tilde{u}\|_{X}\,\|\tilde{v}\|_{Y}.\label{BE2}
\end{equation}
\end{lem}
\noindent{\it Proof.}
 By (\ref{emb2}), it suffices to show that
\begin{eqnarray}
&&\sum_d 2^{-\frac{d}{2}}\left\|m_{-\rho,a}(\xi)
\displaystyle \int_\star \dfrac{f(\xi_1,\tau_1)}{m_{-\rho,a}(\xi_1)\langle\tau_1-\xi_1^2\rangle^{\frac{1}{2}}}
\dfrac{g(-\xi_2,-\tau_2)}{m_{\alpha,a}(\xi_2)\langle\tau_2+\xi_2^2\rangle^\beta}\right\|_{L^2_{\xi\tau}(B_d)}
\nonumber\\
&&\lesssim
\|f\|_{L^2_{\xi\tau}}\,\|g\|_{\hat{X}^{0,0}}\label{6.4}
\end{eqnarray}
for any $f\in L^2(\R^2),g\in \hat{X}^{0,0}$.
we divide (\ref{6.4}) into
two parts to analyze.
$$
\mbox{Part 1. } |\xi_2|\lesssim 1 \mbox{\,\,or\,\,} |\xi_1|\lesssim 1;\quad
\mbox{Part 2. } |\xi_2|\gg 1 \mbox{ and } |\xi_1|\gg 1.
$$

\noindent {\bf Part 1}. $|\xi_2|\lesssim 1 \mbox{\,\,or\,\,} |\xi_1|\lesssim 1.$ It concludes that
$$
|\xi|,|\xi_1|,|\xi_2|\lesssim 1;
\quad \mbox{or} \quad
|\xi|\sim |\xi_1|\gg1 ,|\xi_2|\lesssim  1;
\quad \mbox{or} \quad
|\xi|\sim |\xi_2|\gg1 ,|\xi_1|\lesssim  1.
$$
But they  all can be treated as Part 1 in the proof of Lemma 6.1. Indeed,
when $|\xi|,|\xi_1|,|\xi_2|\lesssim 1$, then $|\xi|\lesssim|\xi_1|$ or $|\xi|\lesssim|\xi_2|$.
So we have
$$
m_{-\rho,a}(\xi)\cdot m_{-\rho,a}(\xi_1)^{-1}\lesssim 1, \quad \mbox{or }\quad
m_{-\rho,a}(\xi)\cdot m_{\alpha,a}(\xi_2)^{-1}\lesssim 1.
$$
When $ |\xi|\sim |\xi_1|\gg1 ,|\xi_2|\lesssim  1$. Then
$$
m_{-\rho,a}(\xi)\cdot m_{-\rho,a}(\xi_1)^{-1}\lesssim 1.
$$
When $|\xi|\sim |\xi_2|\gg1 ,|\xi_1|\lesssim  1$. Then
$$
m_{-\rho,a}(\xi), m_{\alpha,a}(\xi_2)^{-1}\lesssim 1.
$$
Hence, the argument used in Part 1 in the proof of Lemma 6.1 follows (\ref{6.4}) in this part.

\noindent {\bf Part 2}. $|\xi_2|\gg 1 \mbox{ and } |\xi_1|\gg 1.$ We further divide (\ref{6.2}) into
two subparts to analyze.

 {\bf Subpart 1}. $|\xi_1|\gg 1,|\xi_2|\gg 1, |\xi|\gtrsim |\xi_1|$. Then
$$
m_{-\rho,a}(\xi)\cdot m_{-\rho,a}(\xi_1)^{-1},\, m_{\alpha,a}(\xi_2)^{-1}\lesssim 1.
$$
Therefore, by (\ref{XE1}),
\begin{eqnarray*}
&& \left\|m_{-\rho,a}(\xi)
   \displaystyle \int_\star \dfrac{f(\xi_1,\tau_1)}{m_{-\rho,a}(\xi_1)\langle\tau_1-\xi_1^2\rangle^{\frac{1}{2}}}
   \dfrac{g(-\xi_2,-\tau_2)}
   {m_{\alpha,a}(\xi_2)\langle\tau_2+\xi_2^2\rangle^\beta}\right\|_{L^2_{\xi\tau}(B_d)}\\
&\lesssim&
   \left\| \displaystyle \int_\star \dfrac{f(\xi_1,\tau_1)}{\langle\tau_1-\xi_1^2\rangle^{\frac{1}{2}}}
   \dfrac{g(-\xi_2,-\tau_2)}
   {\langle\tau_2+\xi_2^2\rangle^\beta}\right\|_{L^2_{\xi\tau}(B_d)}\\
&=&
   \sup  \displaystyle \int_\ast h_{d}(\xi,\tau) \dfrac{f(\xi_1,\tau_1)}{\langle\tau_1-\xi_1^2\rangle^{\frac{1}{2}}}
   \dfrac{g(-\xi_2,-\tau_2)}
   {\langle\tau_2+\xi_2^2\rangle^\beta}\\
&\lesssim&
   \sup  \|\check{h_{d}}\|_{L^4_{xt}}\,
   \left\|\F_{\xi\tau}^{-1}\left(\dfrac{f}{\langle\tau-\xi^2\rangle^{\frac{1}{2}}}\right)\right\|_{L^4_{xt}}
   \,\|g\|_{L^2_{\xi\tau}}\\
&\lesssim&
   \sup  \|h_{d}\|_{\hat{X}_{0,\frac{3}{8}+}}\,
   \left\|\dfrac{f}{\langle\tau-\xi^2\rangle^{\frac{1}{2}}}\right\|_{\hat{X}_{0,\frac{3}{8}+}}
   \,\|g\|_{L^2_{\xi\tau}}\\
&\lesssim&
   2^{(\frac{3}{8}+)d}\,
   \|f\|_{L^2_{\xi\tau}}
   \,\|g\|_{L^2_{\xi\tau}},
\end{eqnarray*}
where $\displaystyle\int_\ast
=\int_{\stackrel{\xi_1+\xi_2=\xi,}{
\tau_1+\tau_2=\tau}}\,d\xi_1 d\xi_2 d\tau_1 d\tau_2$, and the supremum is over the set $H_{d}=
\{h_{d}(\xi,\tau): \, \|h_{d}\|_{L^2}\leq 1,
\makebox{supp\,}h_{d}\subset B_d\}$. Inserting it into (\ref{6.4}) in the left-hand side, we
get the estimate in this subpart.

{\bf Subpart 2}. $|\xi_1|\gg 1,|\xi_2|\gg 1, |\xi|\ll |\xi_1|$. Then
$|\xi_1|\sim |\xi_2|$, recall $\rho<\dfrac{1}{2}$ and by (\ref{I3}),
\begin{eqnarray*}
&& \left\|m_{-\rho,a}(\xi)
   \displaystyle \int_\star \dfrac{f(\xi_1,\tau_1)}{m_{-\rho,a}(\xi_1)\langle\tau_1-\xi_1^2\rangle^{\frac{1}{2}}}
   \,\dfrac{g(-\xi_2,-\tau_2)}
   {m_{\alpha,a}(\xi_2)\langle\tau_2+\xi_2^2\rangle^\beta}\right\|_{L^2_{\xi\tau}(B_d)}\\
&\lesssim&
   \left\| \displaystyle \int_\star |\xi_2|^\rho
   \dfrac{f(\xi_1,\tau_1)}{\langle\tau_1-\xi_1^2\rangle^{\frac{1}{2}}}
   \,g(-\xi_2,-\tau_2)
   \right\|_{L^2_{\xi\tau}( B_d)}\\
&=&
   \sup_{H_{d}}  \displaystyle \int_\ast |\xi_2|^\rho  \,h_{d}(\xi,\tau)
   \,\dfrac{f(\xi_1,\tau_1)}{\langle\tau_1-\xi_1^2\rangle^{\frac{1}{2}}}
   \,g(-\xi_2,-\tau_2)\\
&\sim&
   \sup_{H_{d}}  I^\rho_3(h_d, \dfrac{f(\xi,\tau)}{\langle\tau-\xi^2\rangle^{\frac{1}{2}}},g^\star)\\
&\lesssim&
   \sup_{H_{d}}  \|h_{d}\|_{\hat{X}_{0,\frac{1}{2}-}}\,
   \left\|\dfrac{f}{\langle\tau-\xi^2\rangle^{\frac{1}{2}}}\right\|_{\hat{X}_{0,\frac{1}{2}-}}
   \,\|g\|_{L^2_{\xi\tau}}\\
&\lesssim&
   2^{(\frac{1}{2}-)d}\,
   \|f\|_{L^2_{\xi\tau}}
   \,\|g\|_{L^2_{\xi\tau}}.
\end{eqnarray*}
Inserting it into (\ref{6.4}) in the left-hand side, we
have the claim.  \hfill$\Box$

\begin{lem}
When $\mbox{\rm supp}\,\tilde{v} \subset \cup_j(A_j\cap B_{\leq
k_0j})$, $\mbox{\rm supp}\,\tilde{u} \subset \cup_j(A_j\cap B_{\geq
k_0j})$, then
\begin{equation}
\left\|\dfrac{1}{\langle\tau-\xi^2\rangle} (\tilde{u}\ast
\tilde{v}^\star)\right\|_{X} \lesssim
\|\tilde{u}\|_{Y}\,\|\tilde{v}\|_{X}.\label{BE3}
\end{equation}
\end{lem}
\noindent{\it Proof.} It's much similar to the proofs of Lemma 6.2.
But one shall using the estimate on $I^\rho_2$ as a substitute of
$I^\rho_3$ in Subpart 2. We omit the details.   \hfill$\Box$

\begin{lem}
When
$\mbox{\rm supp}\,\tilde{u},\,\tilde{v}
\subset \cup_j(A_j\cap B_{\leq k_0j})$, then
\begin{equation}
\left\|\dfrac{1}{\langle\tau-\xi^2\rangle} (\tilde{u}\ast
\tilde{v}^\star)\right\|_{Z} \lesssim
\|\tilde{u}\|_{X}\,\|\tilde{v}\|_{X}.\label{BE4}
\end{equation}
\end{lem}
\noindent{\it Proof.}
It suffices to show that for any $f,g \in \hat{X}^{0,0}$,
\begin{equation}
\left\|\dfrac{1}{\langle\tau-\xi^2\rangle}
   \displaystyle \int_\star \dfrac{f(\xi_1,\tau_1)}{m_{-\rho,a}(\xi_1)\langle\tau_1-\xi_1^2\rangle^{\frac{1}{2}}}
   \dfrac{g(-\xi_2,-\tau_2)}
   {m_{-\rho,a}(\xi_2)\langle\tau_2+\xi_2^2\rangle^{\frac{1}{2}}}\right\|_{Z}
\lesssim
   \|f\|_{\hat{X}^{0,0}}\|g\|_{\hat{X}^{0,0}}.
  \label{6.6}
\end{equation}
We may assume that $|\xi_1|\geq|\xi_2|$ in the integral domain
(it's similar for $|\xi_1|\leq|\xi_2|$). Then we divide (\ref{6.6}) into
three parts to analyze.
$$
\mbox{Part 1. } |\xi_2|\lesssim 1;\quad
\mbox{Part 2. } |\xi_2|\gg 1, |\xi|\sim |\xi_1|;\quad
\mbox{Part 3. } |\xi_2|\gg 1, |\xi|\ll |\xi_1|.
$$

\noindent {\bf Part 1}. $|\xi_2|\lesssim 1$. By using the embedding $X\hookrightarrow Z$
in the left-hand side of (\ref{6.6}), it can
be treated as Part 1 in the proof in Lemma 6.1.

\noindent {\bf Part 2}. $|\xi_2|\gg 1, |\xi|\sim |\xi_1|$. Then, $|\xi|, |\xi_1|\gg 1$. By
 $X\hookrightarrow Z$ and (\ref{emb2}), the left-hand side of (\ref{6.6}) is  bounded by
\begin{equation}
\sum_d 2^{-\frac{d}{2}}\left\|
\displaystyle \int_\star \dfrac{f(\xi_1,\tau_1)}{\langle\tau_1-\xi_1^2\rangle^{\frac{1}{2}}}
\dfrac{|\xi_2|^\rho g(-\xi_2,-\tau_2)}{\langle\tau_2+\xi_2^2\rangle^{\frac{1}{2}}}\right\|_{L^2_{\xi\tau}(B_d)}.
\label{6.7}
\end{equation}
Note that, by (\ref{I3}),
\begin{eqnarray*}
&& \left\|
   \displaystyle \int_\star
   \dfrac{f(\xi_1,\tau_1)}{\langle\tau_1-\xi_1^2\rangle^{\frac{1}{2}}}\,
   \dfrac{|\xi_2|^\rho g(-\xi_2,-\tau_2)}{\langle\tau_2+\xi_2^2\rangle^{\frac{1}{2}}}
   \right\|_{L^2_{\xi\tau}(B_d)}\\
&=&
   \sup_{H_{d}}  I^\rho_3(h_d, \dfrac{f}{\langle\tau-\xi^2\rangle^{\frac{1}{2}}},
   \dfrac{g^\star}{\langle\tau+\xi^2\rangle^{\frac{1}{2}}})\\
&\lesssim&
   \sup_{H_{d}}  \|h_{d}\|_{\hat{X}_{0,\frac{1}{2}-}}\,
   \left\|\dfrac{f}{\langle\tau-\xi^2\rangle^{\frac{1}{2}}}\right\|_{\hat{X}_{0,\frac{1}{2}-}}
   \,\|g\|_{L^2_{\xi\tau}}\\
&\lesssim&
   2^{(\frac{1}{2}-)d}\,
   \|f\|_{L^2_{\xi\tau}}
   \,\|g\|_{L^2_{\xi\tau}},
\end{eqnarray*}
where the set $H_{d}$ is defined in the proof of Lemma 6.2. Inserting it into (\ref{6.7}), we have (\ref{6.6})
in this part.

\noindent {\bf Part 3}. $|\xi_2|\gg 1, |\xi|\ll |\xi_1|$. Then
$|\xi_1|\sim |\xi_2|$. We further split it into two subparts to
analyze.
$$
\mbox{Subpart 1. } |\tau-\xi^2|\gtrsim \max\{|\tau_1-\xi_1^2|,|\tau_2+\xi_2^2|\};\,\,\,
\mbox{Subpart 2. } |\tau-\xi^2|\ll \max\{|\tau_1-\xi_1^2|,|\tau_2+\xi_2^2|\}.
$$
The division is based on the following algebraic identity
$$
\tau-\xi^2=(\tau_1-\xi_1^2)+(\tau_2+\xi_2^2)-2\xi \xi_2,
$$
which implies
\begin{equation}
\max\left\{|\tau-\xi^2|,|\tau_1-\xi_1^2|,|\tau_2+\xi_2^2|\right\}
\gtrsim
|\xi||\xi_2|.
\label{iabc}
\end{equation}

{\bf Subpart 1}. $|\tau-\xi^2|\gtrsim \max\{|\tau_1-\xi_1^2|,|\tau_2+\xi_2^2|\}$.
Then by (\ref{iabc}), we have $|\tau-\xi^2|\gtrsim |\xi||\xi_2| $. By the embedding
$Y\hookrightarrow Z$ and Lemma 2.1, it suffices to show
\bi
\item[{\rm(1)}]
           $\left\|\dfrac{m_{-\rho,a}(\xi)}{\langle\tau-\xi^2\rangle}
           \displaystyle \int_\star \dfrac{f(\xi_1,\tau_1)}
           {m_{-\rho,a}(\xi_1)}\,
           \dfrac{g(-\xi_2,-\tau_2)}
           {m_{-\rho,a}(\xi_2)}\right\|_{L^2_\xi L^1_\tau}
           \lesssim
           \|f\|_{L^2_\xi L^1_\tau}\|g\|_{L^2_\xi L^1_\tau}$;

\item[{\rm(2)}]
          $\left\|\dfrac{m_{\alpha,a}(\xi)}{\langle\tau-\xi^2\rangle}
           \displaystyle \int_\star \dfrac{f(\xi_1,\tau_1)}
           {m_{-\rho,a}(\xi_1)\langle\tau_1-\xi_1^2\rangle^{\frac{1}{2}}}\,
           \dfrac{g(-\xi_2,-\tau_2)}
           {m_{-\rho,a}(\xi_2)\langle\tau_2+\xi_2^2\rangle^{\frac{1}{2}}}\right\|_{\hat{X}_{0,\beta}}
           \lesssim
           \|f\|_{\hat{X}^{0,0}}\|g\|_{\hat{X}^{0,0}}$,
\ei
for reasonable functions $f,g$.

For (1), we need a further division to analyze.
$$
\mbox{(a). } |\xi|\geq 1;\quad
\mbox{(b). } 1\geq |\xi|\geq |\xi_1|^{-1};\quad
\mbox{(c). } |\xi|\leq |\xi_1|^{-1};.
$$

When (a), $|\xi|\geq 1$. Recall that $\rho<\dfrac{1}{2}$,
then the left-hand side of (1) is controlled by
\begin{eqnarray*}
&&\left\|\dfrac{|\xi|^{-\rho}}{\langle\tau-\xi^2\rangle}
           \displaystyle \int_\star |\xi_1|^{2\rho}
           \,f(\xi_1,\tau_1)
           \,g(-\xi_2,-\tau_2)
           \right\|_{L^2_\xi L^1_\tau}\\
&\lesssim&
           \left\||\xi|^{-\rho-1}
           \displaystyle \int_\star |\xi_1|^{2\rho-1}
           \,f(\xi_1,\tau_1)
           \,g(-\xi_2,-\tau_2)\right\|_{L^2_\xi L^1_\tau}\\
&\lesssim&
           \||\xi|^{-\rho-1}(f\ast g^\star)\|_{L^2_\xi L^1_\tau}\\
&\lesssim&
           \||\xi|^{-\rho-1}\|_{L^2_\xi L^\infty_\tau(|\xi|\geq 1)}\,\|f\ast g^\star\|_{L^\infty_\xi L^1_\tau}\\
&\lesssim&
           \|f\|_{L^2_\xi L^1_\tau}\,\|g\|_{L^2_\xi L^1_\tau}.
\end{eqnarray*}

When (b), $1\geq |\xi|\geq |\xi_1|^{-1}$. Note that $2\rho-a\leq \dfrac{1}{2}$, the
left-hand side of (1) is controlled by
\begin{eqnarray}
&&\left\|\dfrac{|\xi|^{a}}{\langle\tau-\xi^2\rangle}
           \displaystyle \int_\star |\xi_1|^{2\rho}\,f(\xi_1,\tau_1)
           \,g(-\xi_2,-\tau_2)\right\|_{L^2_\xi L^1_\tau}\nonumber\\
&\lesssim&
           \left\|\dfrac{1}{\langle\tau-\xi^2\rangle^{1-a}}
           \displaystyle \int_\star |\xi_1|^{2\rho-a}
           \,f(\xi_1,\tau_1)
           \,g(-\xi_2,-\tau_2)\right\|_{L^2_\xi L^1_\tau}\nonumber\\
&\lesssim&
           \sum_{j_1}\sum_{0\geq  j\geq -j_1}\sum_{d\geq j+j_1}2^{(a-1)d}\,2^{\frac{j_1}{2}}
           \|(f_{j_1}\ast g_{j_1}^\star)\|_{L^2_\xi L^1_\tau(\dot{A}_j\cap B_d)}
           \label{6.9},
\end{eqnarray}
where
$$
f_{j_1}(\xi,\tau)=f(\xi,\tau)\chi_{A_{j_1}}(\xi,\tau),
\quad g_{j_1}(\xi,\tau)=g(\xi,\tau)\chi_{A_{j_1}}(\xi,\tau),
$$
and
$$
\dot{A}_j=\left\{(\xi,\tau)\in \R^2: 2^j\leq |\xi|\leq 2^{j+1}\right\}.
$$
Further, recall that $a<\dfrac{1}{2}$, we have
\begin{eqnarray*}
(\ref{6.9})
&\lesssim&
           \sum_{j_1}\sum_{0\geq j\geq -j_1}\sum_{d\geq j+j_1}2^{(a-1)d}\,2^{\frac{j_1}{2}}
           \,\|1\|_{L^2_\xi L^\infty_\tau(\dot{A}_j\cap B_d)}
           \,\|(f_{j_1}\ast g_{j_1}^\star)\|_{L^\infty_\xi L^1_\tau}\\
&\lesssim&
           \sum_{j_1}\sum_{0\geq j\geq -j_1}\sum_{d\geq j+j_1}
           2^{(a-1)d}\,2^{\frac{j_1}{2}} \,2^{\frac{j}{2}}
           \|f_{j_1}\ast g_{j_1}^\star\|_{L^\infty_\xi L^1_\tau}\\
&\lesssim&
           \sum_{j_1}\sum_{0\geq j\geq -j_1}\sum_{d\geq j+j_1}
           2^{(a-1)d}\,2^{\frac{j_1}{2}}\,2^{\frac{j}{2}}
           \|f_{j_1}\|_{L^2_\xi L^1_\tau}\,\|g_{j_1}\|_{L^2_\xi L^1_\tau}\\
&\lesssim&
           \sum_{j_1}\sum_{0\geq j\geq -j_1}\,2^{(a-\frac{1}{2})j_1}
           \,2^{(a-\frac{1}{2})j}\,
           \|f_{j_1}\|_{L^2_\xi L^1_\tau}\,\|g_{j_1}\|_{L^2_\xi L^1_\tau}\\
&\lesssim&
           \sum_{j_1}
           \|f_{j_1}\|_{L^2_\xi L^1_\tau}\,\|g_{j_1}\|_{L^2_\xi L^1_\tau}\\
&\lesssim&
           \|f\|_{L^2_\xi L^1_\tau}\,\|g\|_{L^2_\xi L^1_\tau},
\end{eqnarray*}
where we use the  Cauchy-Schwarz inequality in the last step.

When (c), $|\xi|\leq |\xi_1|^{-1}$. Again, the
left-hand side of (1) is controlled by
\begin{eqnarray*}
&&
           \left\|\dfrac{1}{\langle\tau-\xi^2\rangle}
           \displaystyle \int_\star |\xi_1|^{2\rho-a}\,f(\xi_1,\tau_1)
           \,g(-\xi_2,-\tau_2)\right\|_{L^2_\xi L^1_\tau}\\
&\lesssim&
           \left\|
           \displaystyle \int_\star |\xi_1|^{\frac{1}{2}}\,f(\xi_1,\tau_1)
           \,g(-\xi_2,-\tau_2)
           \right\|_{L^2_\xi L^1_\tau}\\
&\lesssim&
           \sum_{j_1}2^{\frac{j_1}{2}}\,
           \left\|
           \displaystyle \int_\star \,f_{j_1}(\xi_1,\tau_1)
           \,g_{j_1}(-\xi_2,-\tau_2)
           \right\|_{L^2_\xi L^1_\tau(|\xi|\leq 2^{-j_1})}\\
&\lesssim&
           \sum_{j_1}2^{\frac{j_1}{2}}\,
           \|1\|_{L^2_\xi L^\infty_\tau(|\xi|\leq 2^{-j_1})}
           \left\|
           f_{j_1}\ast g_{j_1}^\star
           \right\|_{L^\infty_\xi L^1_\tau}\\
&\lesssim&
           \sum_{j_1}
           \|f_{j_1}\|_{L^2_\xi L^1_\tau}\,\|g_{j_1}\|_{L^2_\xi L^1_\tau}\\
&\lesssim&
           \|f\|_{L^2_\xi L^1_\tau}\,\|g\|_{L^2_\xi L^1_\tau}.
\end{eqnarray*}

For (2). When $|\xi|\leq 1$, then
the left-hand side of (2) is dominated by
\begin{eqnarray}
&        &
           \left\|\dfrac{|\xi|^a}{\langle\tau-\xi^2\rangle^{1-\beta}}
           \displaystyle \int_\star |\xi_1|^{2\rho}\dfrac{f(\xi_1,\tau_1)}
           {\langle\tau_1-\xi_1^2\rangle^{\frac{1}{2}}}\,
           \dfrac{g(-\xi_2,-\tau_2)}
           {\langle\tau_2+\xi_2^2\rangle^{\frac{1}{2}}}\right\|_{L^2_{\xi\tau}}\nonumber\\
&\lesssim&
           \left\|\dfrac{1}{\langle\tau-\xi^2\rangle^{1-a-\beta}}
           \displaystyle \int_\star |\xi_1|^{\frac{1}{2}}\dfrac{f(\xi_1,\tau_1)}
           {\langle\tau_1-\xi_1^2\rangle^{\frac{1}{2}}}\,
           \dfrac{g(-\xi_2,-\tau_2)}
           {\langle\tau_2+\xi_2^2\rangle^{\frac{1}{2}}}\right\|_{L^2_{\xi\tau}}.
           \label{6.10}
\end{eqnarray}
Choosing $\beta$ small enough, such that $1-a-\beta>\dfrac{1}{2}$. Remember that
$|\tau-\xi^2|\gtrsim \max\{|\tau_1-\xi_1^2|,|\tau_2+\xi_2^2|\}$, so we have a crude bound of
(\ref{6.10}) that
\begin{eqnarray*}
&       &
           \left\|\dfrac{1}{\langle\tau-\xi^2\rangle^{\frac{1}{2}+}}
           \displaystyle \int_\star |\xi_1|^{\frac{1}{2}}\dfrac{f(\xi_1,\tau_1)}
           {\langle\tau_1-\xi_1^2\rangle^{\frac{1}{2}+}}\,
           \dfrac{g(-\xi_2,-\tau_2)}
           {\langle\tau_2+\xi_2^2\rangle^{\frac{1}{2}+}}\right\|_{L^2_{\xi\tau}}\\
&\sim&
           \sup_{\|h\|_{L^2}\leq 1}  I^\frac{1}{2}_3\left(\dfrac{h}{\langle\tau-\xi^2\rangle^{\frac{1}{2}+}},
           \dfrac{f}{\langle\tau-\xi^2\rangle^{\frac{1}{2}+}},
           \dfrac{g^\star}{\langle\tau+\xi^2\rangle^{\frac{1}{2}+}}\right)\\
&\lesssim&
           \sup_{\|h\|_{L^2}\leq 1}  \left\|\dfrac{h}{\langle\tau-\xi^2\rangle^{\frac{1}{2}+}}
           \right\|_{\hat{X}_{0,\frac{1}{2}+}}\,
           \left\|\dfrac{f}{\langle\tau-\xi^2\rangle^{\frac{1}{2}+}}\right\|_{\hat{X}_{0,\frac{1}{2}+}}
           \,\|g\|_{L^2_{\xi\tau}}\\
&\lesssim&
           \|f\|_{L^2_{\xi\tau}}
           \,\|g\|_{L^2_{\xi\tau}}.
\end{eqnarray*}
where we use (\ref{3.6}) in the third step.

When $|\xi|\geq 1$, then
the left-hand side of (2) is dominated by
\begin{eqnarray*}
&    &     \left\||\xi|^{\alpha-1+\beta}
           \displaystyle \int_\star |\xi_1|^{2\rho-1+\beta}\dfrac{f(\xi_1,\tau_1)}
           {\langle\tau_1-\xi_1^2\rangle^{\frac{1}{2}}}\,
           \dfrac{g(-\xi_2,-\tau_2)}
           {\langle\tau_2+\xi_2^2\rangle^{\frac{1}{2}}}\right\|_{L^2_{\xi\tau}}\\
&\lesssim&
           \left\|
           \dfrac{f}
           {\langle\tau-\xi^2\rangle^{\frac{1}{2}}}\ast
           \dfrac{g^\star}
           {\langle\tau+\xi^2\rangle^{\frac{1}{2}}}\right\|_{L^2_{\xi\tau}}\\
&\lesssim&
           \left\|\F_{\xi\tau}^{-1}\left(\dfrac{f}{\langle\tau-\xi^2\rangle^{\frac{1}{2}}}\right)\right\|_{L^4_{xt}}
           \,\left\|\F_{\xi\tau}^{-1}\left(\dfrac{g^\star}{\langle\tau+\xi^2\rangle^{\frac{1}{2}}}\right)\right\|_{L^4_{xt}}\\
&\lesssim&
           \left\|\dfrac{f}{\langle\tau-\xi^2\rangle^{\frac{1}{2}}}\right\|_{\hat{X}_{0,\frac{3}{8}+}}
           \,\left\|\dfrac{g}{\langle\tau-\xi^2\rangle^{\frac{1}{2}}}\right\|_{\hat{X}_{0,\frac{3}{8}+}}\\
&\lesssim&
          \|f\|_{L^2_{\xi\tau}}
          \,\|g\|_{L^2_{\xi\tau}},
\end{eqnarray*}
where we choosing $\beta$ small enough again, such that $2\rho-1+\beta\leq 0$.

{\bf Subpart 2}. $|\tau-\xi^2|\ll \max\left\{|\tau_1-\xi_1^2|,|\tau_2+\xi_2^2|\right\}$. Then by (\ref{iabc}), we have
\begin{equation}
|\tau_1-\xi_1^2|=\max\left\{|\tau_1-\xi_1^2|,|\tau_2+\xi_2^2|\right\}\gtrsim |\xi||\xi_2|
\label{6.11}
\end{equation}
or
$$
|\tau_2+\xi_2^2|=\max\left\{|\tau_1-\xi_1^2|,|\tau_2+\xi_2^2|\right\}\gtrsim
|\xi||\xi_2|.
$$
We just consider the case (\ref{6.11}) (the other is similar).
By the embedding $X\hookrightarrow Z$ and (\ref{emb2}),
the left-hand side of (\ref{6.6}) is dominated by
\begin{eqnarray}
&         &
     \sum_d 2^{-\frac{d}{2}}\left\|m_{-\rho,a}(\xi)
     \displaystyle \int_\star |\xi_2|^{2\rho}
     \dfrac{f(\xi_1,\tau_1)}{\langle\tau_1-\xi_1^2\rangle^{\frac{1}{2}}}
     \dfrac{g(-\xi_2,-\tau_2)}{\langle\tau_2+\xi_2^2\rangle^{\frac{1}{2}}}
     \right\|_{L^2_{\xi\tau}(B_d)}
     \nonumber\\
&\lesssim &
      \sum_d 2^{-\frac{d}{2}}\left\|m_{-\rho,a}(\xi)|\xi|^{-a}
      \displaystyle \int_\star |\xi_2|^{2\rho-a}
      \dfrac{f(\xi_1,\tau_1)}{\langle\tau_1-\xi_1^2\rangle^{\frac{1}{2}-a}}
      \dfrac{g(-\xi_2,-\tau_2)}{\langle\tau_2+\xi_2^2\rangle^{\frac{1}{2}}}
      \right\|_{L^2_{\xi\tau}(B_d)}
      \nonumber\\
&\lesssim &
      \sum_d 2^{-(\frac{1}{2}+)d}
      \left\|
      \displaystyle \int_\star |\xi_2|^{\frac{1}{2}}\,
      f(\xi_1,\tau_1)
      \,\dfrac{g(-\xi_2,-\tau_2)}{\langle\tau_2+\xi_2^2\rangle^{\frac{1}{2}+}}
      \right\|_{L^2_{\xi\tau}(B_d)}
      \label{6.12}
\end{eqnarray}
by the fact $a<\dfrac{1}{2}$, and
$|\tau_1-\xi_1^2|\geq \max\left\{|\tau-\xi^2|,|\tau_2+\xi_2^2|\right\}$.
Since
\begin{eqnarray*}
&       &
           \left\|
           \displaystyle \int_\star |\xi_2|^{\frac{1}{2}}\,
           f(\xi_1,\tau_1)
           \,\dfrac{g(-\xi_2,-\tau_2)}{\langle\tau_2+\xi_2^2\rangle^{\frac{1}{2}+}}
           \right\|_{L^2_{\xi\tau}(B_d)}\\
&\sim&
           \sup_{H_{d}}  I^\frac{1}{2}_2\left(h_{d},
           f,\dfrac{g^\star}{\langle\tau+\xi^2\rangle^{\frac{1}{2}+}}\right)\\
&\lesssim&
           \sup_{H_{d}}  \left\|h_{d} \right\|_{\hat{X}_{0,\frac{1}{2}+}}
           \,\|f\|_{L^2_{\xi\tau}}\,
           \left\|\dfrac{g}{\langle\tau-\xi^2\rangle^{\frac{1}{2}+}}\right\|_{\hat{X}_{0,\frac{1}{2}+}}\\
&\lesssim&
           2^{(\frac{1}{2}+)d}\|f\|_{L^2_{\xi\tau}}
           \,\|g\|_{L^2_{\xi\tau}},
\end{eqnarray*}
where the set $H_{d}$ is defined in the proof of Lemma 6.2.
Inserting it into (\ref{6.12}), we obtain (\ref{6.6}).
This completes the proof of the lemma.   \hfill$\Box$

Combining (\ref{BE1}), (\ref{BE2}),  (\ref{BE3}) and (\ref{BE4}), we
establish (\ref{BE}), and hence finish the proof of  Theorems 1.2
and 1.4.

\end{document}